\documentclass[10pt]{amsart}
\usepackage{amssymb, tikz, mathrsfs, stmaryrd}
\usepackage[all]{xy}

\title[Uniformization of certain correspondences]{$p$-adic uniformization and the action of Galois on certain affine correspondences}
\author{Patrick Ingram}
\date{\today}
\email{pingram@yorku.ca}
\address{Colorado State University, Fort Collins, USA\newline Current address: York University, Toronto, Canada}

\thanks{This research was supported in part by Simons Collaboration Grant \#283120. }


\newcommand{\QQ}{\mathbb{Q}}
\newcommand{\ZZ}{\mathbb{Z}}
\newcommand{\CC}{\mathbb{C}}

\newcommand{\PP}{\mathbb{P}}

\newcommand{\Ocal}{\mathcal{O}}

\newcommand{\Gal}{\operatorname{Gal}}
\newcommand{\Spec}{\operatorname{Spec}}

\newcommand{\Aut}{\operatorname{Aut}}

\newcommand{\MOD}[1]{~(\textup{mod}~#1)}

\newcommand{\sep}[1]{\overline{#1}}

\newcommand{\path}{\mathscr{P}}

\renewcommand{\phi}{\varphi}
\renewcommand{\epsilon}{\varepsilon}

\newtheorem{theorem}{Theorem}
\newtheorem{lemma}[theorem]{Lemma}

\theoremstyle{remark}
\newtheorem{remark}[theorem]{Remark}

\theoremstyle{definition}

\begin{document}
\begin{abstract}
Given two monic polynomials $f$ and $g$ with coefficients in a number field $K$, and some $\alpha\in K$, we examine the action of the absolute Galois group $\Gal(\overline{K}/K)$ on the directed graph of iterated preimages of $\alpha$ under the correspondence $g(y)=f(x)$, assuming that $\deg(f)>\deg(g)$ and that $\gcd(\deg(f), \deg(g))=1$. If a prime of $K$ exists at which $f$ and $g$ have integral coefficients, and at which $\alpha$ is not integral, we show that this directed graph of preimages consists of finitely many $\Gal(\overline{K}/K)$-orbits. We obtain this result by establishing a $p$-adic uniformization of such correspondences, tenuously related to B\"{o}ttcher's uniformization of polynomial dynamical systems over $\CC$, although the construction of a B\"{o}ttcher coordinate for complex holomorphic correspondences remains unresolved.
\end{abstract}

\maketitle


\section{Introduction}
Let $K$ be a number field, and let $f(z)\in K[z]$ be a polynomial. Boston and Jones~\cite{bostonjones} considered the action of the absolute Galois group $\Gal(\overline{K}/K)$ on the set of iterated preimages of any $\alpha\in K$ under $f$. These preimages form a tree $T_{f, \alpha}$, with edge relation defined by the application of $f$, and so we  have an \emph{arboreal Galois representation}
\[\rho_{f, \alpha}:\Gal(\sep{K}/K)\to \operatorname{Aut}(T_{f, \alpha}),\]
where the latter is the automorphism group of $T_{f, \alpha}$ as an abstract graph with a marked point $\alpha$. One naturally wonders when the image of $\rho_{f, \alpha}$ has finite index in $\operatorname{Aut}(T_{f, \alpha})$, or perhaps (less ambitiously) when the paths through $T_{f, \alpha}$ break down into only finitely many orbits under $\Gal(\overline{K}/K)$ (see Jones and Levy~\cite{joneslevy}, and for a broader overview, the survey article of Jones~\cite{jones}).

Now let $g(z)\in K[z]$ be another polynomial, and consider the \emph{correspondence}
\begin{equation}\label{eq:cdef}C:g(y)=f(x).\end{equation}
That is, consider the (directed) graph whose vertices are the elements of $\overline{K}$ with an edge from $x$ to $y$ if and only if $g(y)=f(x)$. For any $\alpha\in K$, let $P_{C, \alpha}$ denote the subgraph of backward paths from $\alpha$. Then again, $\Gal(\overline{K}/K)$ acts on $P_{C, \alpha}$, and we are naturally impelled to inquire as to how free this action is. One would like to conjecture that there are in general few constraints, other than the edge relation defined by the correspondence. In other words, one might conjecture that the image of $\Gal(\overline{K}/K)$ has finite index in the graph-theoretic automorphism group $\operatorname{Aut}(P_{C, \alpha})$, except for some easy-to-describe exceptional class of correspondences. Note that this case contains the previous one, since setting $g(y)=y$ gives $P_{C, \alpha}=T_{f, \alpha}$. If instead we set $f(x)=x$ and leave $g$ alone, then $P_{C, \alpha}$ is a single $K$-rational path (consisting of the iterated \emph{post}images of $\alpha$ under $g$). In this case $\Gal(\overline{K}/K)$ surjects onto $\operatorname{Aut}(P_{C, \alpha})$ in a fairly unexciting way.

\begin{figure}
\caption{Two backward paths from $\alpha$ when $\deg(f)=3$, $\deg(g)=2$}
\vspace{2mm}
\begin{tikzpicture}[->,shorten >=1pt,auto,node distance=1.5cm,
  thick,main node/.style={circle,fill=white,draw,font=\sffamily\Large}]
  
  \draw[-]  plot [smooth, tension=1] coordinates { (-3,.25) (0,0) (3,.125) (4,0)};
  \node at (4.5, 0){$\PP^1$};

  \draw[-]  plot [smooth, tension=1] coordinates { (-3,5.5) (0,5.25) (3,5.375) (4,5.25)};
  \node at (4.5, 5.25){$\check{\path}_{C, \alpha}$};

  \node[main node] (1) {$\alpha$};
  \node[main node] (2) [above of=1] {$\alpha_1$};
  \node[main node] (215) [dashed] [above of=2] {};
  \node[main node] (21) [dashed]  [right of=215] {};
  \node[main node] (22) [dashed] [right of=21] {};
  \node[main node] (3) [left of=215] {$\alpha_2$};
  
  \node[main node] (98) [dashed] [right of=2] {};
  \node[main node] (99) [dashed] [right of=98] {};
    
    
  \node[main node] (10) [dashed]  [right of=1] {};
  \node[main node] (11) [dashed]  [left of=2] {};
  
  
  \node[main node] (31)   [above left of=3] {$\alpha_3$};
  \node[main node] (32)  [above right of=3] {$\beta_3$};
  
  \node[main node] (33) [right of=32] [dashed] {};
  
  \node[main node] (41) [above of=31] {$P$};
  \node[main node] (42) [above of=32] {$Q$};

  \path[every node/.style={font=\sffamily\small}]
   (2) edge node [right] {} (1)
   (3) edge node [right] {} (2)
   (2) edge [dashed] node [right] {} (10)
   (3) edge [dashed] node [right] {} (11)

   (98) edge [dashed] node [right] {} (1)
   (98) edge [dashed] node [right] {} (10)
   (99) edge [dashed] node [right] {} (1)
   (99) edge [dashed] node [right] {} (10)

   (31) edge [dashed] node [right] {} (215)
   (32) edge [dashed] node [right] {} (215)

   (33) edge [dashed] node [right] {} (215)
   (33) edge [dashed] node [right] {} (3)

   (21) edge [dashed] node [right] {} (2)
   (21) edge [dashed] node [right] {} (11)
   (22) edge [dashed] node [right] {} (11)
   (31) edge  node [right] {} (3)
   (32) edge  node [right] {} (3)
   (41) edge [dashed] node [right] {} (31)
   (42) edge [dashed] node [right] {} (32)
   (22) edge [dashed] node [right] {} (2);


\end{tikzpicture}
\end{figure}
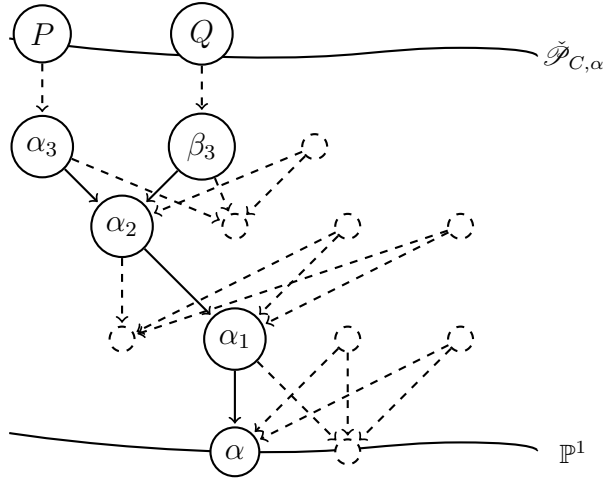

The purpose of this note is to extend some earlier work of the author~\cite{blms} on arboreal Galois representations over local fields to the setting of correspondences.  If we start over a number field and there exists a prime such that the conditions of the below theorem are met at that localization, then of course the conclusion holds over that number field as well.
In this context it is arguably more natural to think in terms of forward orbits, but for the sake of consistency we consider preimages. Moving from a correspondence
$C:g(y)=f(x)$
to its \emph{dual}
$\check{C}:f(y)=g(x)$
swaps the two notions. We will say that a correspondence as in~\eqref{eq:cdef} is \emph{polarized} if and only if $\deg(g)<\deg(f)$ (which is the definition in \cite{cor_heights} restricted to the current context). If $C$ is defined over a local ring, then we will say that $C$ has \emph{good reduction} if and only if the coefficients of $f$ and $g$ are integral, and the leading coefficient of each is a unit.  Denote by $\check{\path}_{C, \alpha}$ the set of paths through $P_{C, \alpha}$, that is, the set of sequences $\{\alpha_n\}_{n\geq 0}$ with $\alpha_0=\alpha$, and $g(\alpha_{n})=f(\alpha_{n+1})$, an object acted upon naturally by $\operatorname{Gal}(\sep{K}/K)$, if $f$ and $g$ have coefficients in $K$. Finally, if $G$ is a group acting on a set $X$, we will say that $G$ acts \emph{nearly transitively} if and only if $X$ is a union of finitely many $G$-orbits.
\begin{theorem}\label{th:main}
Let $K$ be a field of characteristic zero with a discrete, non-archimedean valuation, let $C:g(y)=f(x)$ be a polarized correspondence over $K$, and let $\alpha\in K$. Assume that
\begin{enumerate}
\item $C$ has good reduction;
\item $\deg(f)$ and $\deg(g)$ are units in $K$;
\item $\gcd(\deg(f), \deg(g))=1$; and
\item $\alpha$ is not integral.
\end{enumerate}
Then  $\operatorname{Gal}(\sep{K}/K)$ acts nearly transitively on $\check{\path}_{C, \alpha}$.
\end{theorem}

More specifically, let $T$ be the regular $d$-branching tree whose vertices are the elements of $\bigcup_{n\geq 0} \ZZ/d^n\ZZ$ (a disjoint union), with $g_{n+1}\in \ZZ/d^{n+1}\ZZ$ sitting above $g_n\in \ZZ/d^n\ZZ$ in the tree if and only if $g_{n+1}\equiv g_n\MOD{d^n}$. The \emph{ends} of the tree $T$ (that is, the infinite paths beginning at the root) are then naturally identified with the elements of the group $H=\varprojlim \ZZ/d^n\ZZ$. The elements of $H$ are sequences $(h_0, h_1, ...)$ with $h_n\in \ZZ/d^n\ZZ$, subject to the condition that $h_{n+1}\equiv h_n\MOD{d^n}$. Note that there is a natural action of $H$ on $T$, with $(h_0, h_1, ...)$ taking $g_n\in \ZZ/d^n\ZZ\subseteq T$ to $h_n+g_n$, which is transitive on each level just because $H$ is a group; the corresponding action on the ends of $T$ is just the translation action of $H$ on itself. Any subgroup of $\operatorname{Aut}(T)$ which is conjugate to $H$ will be called a \emph{$d$-adic subgroup}, a condition which is independent of the choice of labelling, and hence may be applied to any infinite $d$-branching rooted tree. In other words, a $d$-adic subgroup is one which arises as $H$ for some labelling of the vertices of $T$ as above.

The next theorem sheds some more light on exactly what sorts of actions $\Gal(\overline{K}/K)$ can have on preimage trees for correspondences.
\begin{theorem}\label{th:second}
Let $K$, $C$, and $\alpha$ be as in Theorem~\ref{th:main}, and let $G\subseteq \Aut(P_{C, \alpha})$ be the image of $\Gal(\overline{K}/K)$. Then there is a $d$-adic subgroup $H\subseteq \operatorname{Aut}(P_{C, \alpha})$ such that $G\cap H$ has finite index in $H$.
\end{theorem}

Note that Theorem~\ref{th:second} immediately implies Theorem~\ref{th:main}, since the $H$-orbits in $\check{\path}_{C, \alpha}\cong H$ are unions of cosets of $G\cap H$, and Theorem~\ref{th:second} implies that there are only finitely many of these cosets. We note also that the results in this context are weaker than the analogous results for single-valued polynomial dynamical systems, as found in~\cite{blms}. This is primarily because the proof in the case of correspondences requires us to pass to much larger extension of $K$. We note also that it might be possible obtain Theorem~\ref{th:main} by a more direct argument, relying on a generalized Eisenstein condition, as in the results of Jones and Levy~\cite{joneslevy}. We were motivated as much by the application as by a wish to generalize the B\"{o}ttcher uniformization to correspondences; it would be interesting to complete the construction of such a coordinate over $\CC$ (the work here gives only a formal uniformization in the archimedean case, with no claim of convergence).

Our approach is similar to that taken in~\cite{blms}, although things become somewhat more abstract in the current setting. In broad strokes, we  complete  $K$ and then construct a partial isomorphism between the correspondence $C:g(y)=f(x)$ and the correspondence $y^{\deg(g)}=x^{\deg(f)}$, generalizing the classical B\"{o}ttcher coordinate~\cite[p.~90]{milnor}. The natural Galois-equivariance of this isomorphism allows us to reduce many questions about backward orbits in $C$ to questions about Kummer theory.

More accurately, though, our isomorphism is between paths in the first correspondence, and paths in the second.
As in~\cite{cor_heights, cor_pcc} we replace our dynamical correspondence $C:g(y)=f(x)$ with a single-valued dynamical system $\sigma_C:\path_C\to \path_C$. Here $\path_C$ is the collection of all possible forward paths through the correspondence (over $\sep{K}$), and $\sigma_C$ is the shift map, which forgets the first vertex of a path, and treats the second vertex as the starting point. One might think of this as a dynamical system in the category of sets, or note that $\sigma_C$ is in fact a morphism of $K$-schemes. Next, we construct a formal isomorphism of dynamical systems
\[\xymatrix{
\path_C \ar[d]_{\Xi} \ar[r]^{\sigma_C}  &  \path_C \ar[d]^\Xi\\  
\path_B \ar[r]_{\sigma_B} & \path_B,
}\]
with  $B:y^{\deg(g)}=x^{\deg(f)}$. In the case of good reduction, our formal isomorphism converges to an actual isomorphism exactly on the set of paths whose initial vertex is non-integral (i.e., on a disk of radius one centred at infinity). We make the fairly simple connection between inverse images in the system $\sigma_C:\path_C\to\path_C$ and inverse images relative to the original correspondence $C$, and then use the relation to the simpler correspondence $B$ to say something about the action of Galois.

Note that the hypothesis $\gcd(\deg(f), \deg(g))=1$ in  Theorem~\ref{th:main} is necessary for our approach to work. For example, $\operatorname{Gal}(\overline{K}/K)$ cannot act nearly transitively on inverse orbits under the correspondence $y^2=x^4$. If $\alpha\in K$, then the $n$th preimages of $\alpha$ under this correspondence are those elements of $\overline{K}$ of the form $\zeta \beta$ with $\beta^{2^n}=\alpha$ and $\zeta^{4^n}=1$. If $\sigma\in \operatorname{Gal}(\overline{K}/K)$ satisfies $\sigma(\beta)=\zeta\beta$, then
\[\beta^{2^n}=\alpha=\sigma(\alpha)=\sigma(\beta)^{2^n}=\zeta^{2^n}\beta^{2^n},\]
and so $\zeta^{2^n}=1$. In other words, there are at least as many Galois orbits amongst the $n$th preimages of $\alpha$ as there are cosets of $4^n$th roots of unity by $2^n$th roots, in other words at least $2^n$ orbits. This specific pathology is removed by requiring that $f$ and $g$ are not of the form $f(x)=F\circ h(x)$ and $g(y)=G\circ h(y)$ for any non-linear polynomial $h$, but since we are uniformizing by the correspondence $y^{\deg(g)}=x^{\deg(f)}$, this ultimately requires us to assume that $\deg(g)$ and $\deg(f)$ have no common factor.

This paper is organized as follows. In Section~\ref{sec:formal} we construct the formal isomorphism $\Xi$, without regard to convergence. This construction is very general, and is largely independent of the nature of the ground field. In Section~\ref{sec:local} we then show that, if the ground field is a local field, this formal isomorphism gives rise to an actual isomorphism in a neighbourhood of infinity, in good reduction and in bad. We have left unresolved the question of convergence when the ground field is $\CC$, which would be of interest for the purpose of studying the complex dynamics of correspondences. Finally, in Section~\ref{sec:galois} we examine the Galois theory of the correspondence $y^e=x^d$, and what it tells us about other correspondences.


\section{The formal series}\label{sec:formal}

Let $R$ be a Dedekind ring, with field of fractions $K$, and let $1\leq e<d$ be two integers invertible in $R$ (so in particular the characteristic of $K$ does not divide $ed$). We fix two monic polynomials $f(x), g(x)\in R[x]$, satisfying $\deg(f)=d$ and $\deg(g)=e$, and consider the (polarized) correspondence
\[C:g(y)=f(x).\]
 Although the construction is outlined in \cite{cor_heights}, we briefly recall how the scheme $\path_C$ is constructed, omitting the subscript when context makes it clear. Our aim is to construct a scheme $\path$, a finite morphism $\sigma:\path\to\path$, and a morphism $\pi:\path\to\PP^1$ which make the following diagram commute.
\begin{equation}\label{eq:comdi}\xymatrix{
\path \ar[d]_{\pi} \ar[dr]_{\epsilon} \ar[rr]^\sigma &  &  \path \ar[d]^\pi\\  
\PP^1 & C \ar[l]^x \ar[r]_y & \PP^1.
}\end{equation}
Here $\epsilon$ picks out the first edge in a path, the point on $C$ corresponding to the edge from $\pi(P)$ to $\pi\circ\sigma(P)$.
 
Let
\[A=R[x_{i, j}:0\leq i<\infty, 0\leq j\leq 1]\]
be graded so that $x_{i,j}$ has weight $(d/e)^i$. Note that our hypothesis $d>e$ implies that $A$ contains only finitely many monomials below any given weight, which is essential to the inversion of power series below. Let $I$ be the (homogeneous) ideal generated by polynomials of the form
\[x_{i+1, 0}^{\deg(g)}x_{i, 0}^{\deg(f)}\left(g(x_{i+1, 1}/x_{i+1, 0})-f(x_{i, 1}/x_{i, 0})\right)\]
We then set $\mathscr{P}=\operatorname{Proj}(A/I)$. The map $\pi:\mathscr{P}\to \PP^1$ is that induced by the natural inclusion $R[x_0, x_1]\to A$ by $x_j\mapsto x_{0, j}$, while the map $\sigma:\mathscr{P}\to\mathscr{P}$ corresponds to $A\to A$ by $x_{i, j}\mapsto x_{i+1, j}$.
 In the special case where $g(y)=y^e$ and $f(x)=x^d$, we denote this scheme by $\mathscr{B}$, with corresponding maps $\psi:\mathscr{B}\to \PP^1$ and $\tau:\mathscr{B}\to\mathscr{B}$. 
 
Writing $\Ocal_{\mathscr{P}, \infty}$ and $\Ocal_{\mathscr{B}, \infty}$ for the completed local rings of these two schemes at the unique point above $\infty\in \PP^1$, we set $\widehat{\mathscr{P}}=\Spec(\Ocal_{\mathscr{P}, \infty})$ and $\widehat{\mathscr{B}}=\Spec(\Ocal_{\mathscr{B}, \infty})$. Note that the maps $\sigma$ and $\tau$ induce endomorphisms of $\widehat{\mathscr{P}}$ and $\widehat{\mathscr{B}}$, respectively, which we will also denote by $\sigma$ and $\tau$.  Our next result makes precise our formal isomorphism of $\mathscr{P}$ and $\mathscr{B}$ at infinity.

\begin{theorem}\label{th:formal}
There is an isomorphism $\Xi:\widehat{\mathscr{P}}\to\widehat{\mathscr{B}}$ of $K$-schemes such that $\Xi\circ\sigma=\tau\circ \Xi$.
\end{theorem}

\begin{proof}
To maintain a level of generality that will be useful in Section~\ref{sec:local}, we suppose for now that we have sequences $g_i, f_i$ of monic polynomials with coefficients in $R$, with $\deg(g_i)=e$ and $\deg(f_i)=d$ for all $i$, and we let
\[S_n=R\llbracket w_0, ..., w_n\rrbracket /\left(w_{i+1}^{e}w_i^{d}(g_{i+1}(w_{i+1}^{-1})-f_i(w_i^{-1})):0\leq i<n\right),\]
noting that $S_n\subseteq S_{n+1}$. Each of these local rings comes equipped with a valuation which is trivial on $R$, that is, nonzero elements of $R$ have valuation zero. By abuse of notation we denote all of these valuations by $v$, normalized so that $v(w_0)=1$. For every $X>0$ we define an ideal \[\mathfrak{m}_{n, X}=\{s\in S_n:v(s)>X\}\cup \{0\}.\] For any ideal $I$, we write $x=y+O(I)$ to indicate that $x-y\in I$.

Now, $w_n^{e^n}=w_0^{d^n}+O(\mathfrak{m}_{n, d^n})$ so, since $S_n$ is Henselian, there is an element $\xi_n=w_0+O(\mathfrak{m}_{n, 1})$ with $\xi_n^{d^n}=w_n^{e^n}$.

Note that if $v(x)<0$ and
\[y^e+\cdots = g_{i+1}(y)=f_i(x)=x^{d}+\cdots\]
with coefficients in $R$, then
\[\left|\frac{y^e}{x^d}-1\right|_v\leq \left|x^{-1}\right|_v.\]
Also, note that if $z=1+O(\mathfrak{m}_{n, 0})$, then $|1-z|_v=|1-z^n|_v$ for any $n\neq 0$ in $R$ (that is, $n$ not divisible by the characteristic of $K$).
From these we estimate, in the $v$-adic absolute value,
\begin{eqnarray*}
\left|1-\frac{\xi_{n}}{\xi_{n+1}}\right|_v&=&\left|1-\frac{\xi_{n}^{d^{n+1}}}{\xi_{n+1}^{d^{n+1}}}\right|_v\\
&=&\left|1-\frac{w_{n}^{de^{n}}}{w_{n+1}^{e^{n+1}}}\right|_v\\
&=&\left|1-\frac{w_n^d}{w_{n+1}^e}\right|_v\\
&\leq &|w_{n}|_v\\
&=&|w_0|_v^{(d/e)^n},
\end{eqnarray*}
whereupon
\[|\xi_{m}-\xi_n|_v\leq  |w_0|_v^{(d/e)^{\min\{m, n\}}+1}.\]
So the sequence $\{\xi_n\}_{n\geq 0}$ is Cauchy (given our hypothesis that $d>e$). If $S$ is the completion of the direct limit $\varinjlim S_n$, then the sequence $\{\xi_n\}_{n\geq 0}$ has a limit $\xi\in S$.

Now restrict attention to the case $g_i=g$ and $f_i=f$ for all $i$, in which case $S=\Ocal_{\path, \infty}$.  Note that for each $n$,
\[(\xi_n\circ \sigma)^{ed^n}=w_{n+1}^{e^{n+1}}=\xi_{n+1}^{d^{n+1}},\]
and so $(\xi_n\circ\sigma)^e/\xi_{n+1}^d=\zeta$ for some $\zeta^{d^n}=1$. On the other hand, we have $v(\xi_n\circ\sigma -w_1)>1$, and $v(\xi_n^d-w_1)>1$, so it must be the case that $\zeta=1$ (or else we have $v(1-\zeta)<0$ for some $\zeta\neq 1$).  It follows that
\[(\xi_n\circ\sigma)^e=\xi_{n+1}^d.\]
Taking limits of both sides as $n\to \infty$, we have $(\xi\circ \sigma)^e=\xi^d$. In other words, the map 
\[\Xi(P)=(\xi(P), \xi\circ\sigma(P), \xi\circ\sigma^2(P), ...)\]
is a morphism of $K$-schemes, $\Xi:\widehat{\path}\to \widehat{\mathscr{B}}$. From the definition, it is clear that $\Xi\circ \sigma =\tau\circ \Xi$.

It remains to show that $\Xi$ is invertible, which we do by a generalization of the Lagrange inversion formula. Let $M$ denote the set of monomials in the $w_i$, and let $N$ denote the set of monomials in variables $z_i$ subject to $z_{i+1}^e=z_i^d$. We may write $\xi$ as
\[\xi=\sum_{w\in M} c(w)w,\]
where $c(w)\in R$. By the standard abuse of notation, we consider these monomials both as elements of $R$, and as functions. Note that if $v$ is the normalized valuation on $S$ induced by the (compatible) valuations $v$ on the $S_n$, then $v(w)\geq 1$ for all $w\in M$, with equality just in case $w=w_0$. Similarly, there is a natural valuation $v'$ on the ring $R[z:z\in N]$ satisfying $v'(z_i)=(d/e)^i$ for all $i\geq 0$, and $v'(r)=0$ for nonzero $r\in R$. 

Note that   $c(w_0)=1$, since $v(\xi-w_0)>1$. 
Write \[\gamma=\sum_{z\in N} b(z)z,\] with the $b(z)$ to be chosen later.
For any $X$ there exist only finitely many $z\in N$ with $v'(z)\leq X$, and we let $\mathfrak{n}_X$ be the ideal generated by monomials $z\in N$ with $v'(z)>X$.
Note that $\{v'(z):z\in N\}$ is well-ordered, and so we may define the coefficients $b(z)$ recursively. We will do so in such a way that
\[\xi(\gamma, \gamma\circ\tau, ...)=z_0+O(\mathfrak{n}_X)\]
for all $X$, i.e., such that $\xi(\gamma, \gamma\circ\tau, ...)=z_0$.
To start, we choose $b(z_0)=1$, which ensures that $\xi(\gamma, \gamma\circ\tau, ...)=z_0+O(\mathfrak{n}_1)$.

Now, assuming that we have chosen $b(z)$ for all $z\in N$ with $v'(z)<X$, assuming that there exists a monomial of valuation $X$, and assuming that for any $Y<X$ we have
\[\xi(\gamma, \gamma\circ\tau, ...)=z_0+O(\mathfrak{n}_Y),\]
we show how to choose $b(z)$ for every monomial $z\in N$ satisfying $v(z)=X$. 
We have
\[\xi(\gamma, \gamma\circ \tau, ...)=\sum_{w\in M}c(w)w(\gamma, \gamma\circ\tau, ...),\]
and we consider each summand separately. The summand corresponding to $w=w_0$ is simply
\[\gamma=\sum_{z\in N}b(z)z,\]
and within this the summands of valuation $X$ are simply those of the form $b(z)z$ with $v'(z)=X$.

In general, the summand corresponding to the monomial $w=w_0^{e_0}w_1^{e_1}\cdots w_{k}^{e_k}$ is
\begin{equation}\label{eq:summands}c(w)\gamma^{e_0}(\gamma\circ\tau)^{e_1}\cdots(\gamma\circ\tau^k)^{e_k}= c(w)\sum_{i=0}^k\sum_{n_{i, 1}, ..., n_{i, e_i}\in N}\prod_{i=0}^k\prod_{j=1}^{e_i}b(n_{i, j})n_{i, j}^{\tau^i}\end{equation}
where the action of $\tau$ on $N$ is defined by $z_i^\tau=z_{i+1}$. Note that
\[v'\left(\prod_{i=0}^k\prod_{j=1}^{e_i}b(n_{i, j})n_{i, j}^{\tau^i}\right)=\sum_{i=0}^k\sum_{j=1}^{e_i}\left(\frac{d}{e}\right)^iv'(n_{i, j}),\]
and that $v'(n_{i, j})\geq 1$ for every monomial. So if an individual summand
\[c(w)\prod_{i=0}^k\prod_{j=1}^{e_i}b(n_{i, j})n_{i, j}^{\tau^i}\]
in~\eqref{eq:summands} has valuation $X$, then either $w=w_0$ (in which case the summand is just $b(z)z$ for some $z$) or else each $n_{i, j}$ satisfies $v'(n_{i, j})<X$. In other words, for each monomial $z\in N$ with $v'(z)=X$, the coefficient of $z$ in  $\xi(\gamma, \gamma\circ \tau, ...)$ is $b(z)+\beta(z)$, where $\beta(z)$ is some $R$-linear combination of products of coefficients of the form $b(z')$ where $v'(z')<X$ (i.e., whose values have previously been chosen). We may now simply choose $b(z)=-\beta(z)$.
 This choice gives us
\[\xi(\gamma, \gamma\circ\tau, ...)=z_0+O(\mathfrak{n}_X),\]
and, taking $X\to\infty$, we have our series $\gamma$.

Now, define $\Psi:\widehat{\mathscr{B}}\to \widehat{\mathscr{P}}$ by
\[\Psi(Q)=(\gamma(Q), \gamma\circ\tau(Q), ...).\]
We claim  $\Psi$ and $\Xi$ are inverses. Note that
\begin{eqnarray*}
\Xi\circ\Psi(z_0, z_1, ...)&=&(\xi(\gamma, \gamma\circ\tau, ...), \xi(\gamma\circ\tau, \gamma\circ\tau^2, ...))\\
&=&(z_0, \xi(\gamma(\tau(z_0, ...)), \gamma\circ\tau(\tau(z_0, ...), ...)))\\
&=&(z_0, z_1, ...),
\end{eqnarray*}
so $\Psi$ is a right inverse of $\Xi$.
On the other hand the same argument that constructs $\Psi$ also constructs a right inverse of $\Psi$, which then must be $\Xi$.
\end{proof}


\section{Convergence over a local ring}\label{sec:local}

In Section~\ref{sec:formal} we constructed an isomorphism of formal schemes. In this section we will take $K$ to be a separably closed field in which $ed$ is a unit, complete with respect to a non-archimedean absolute value. We will show that there is a neighbourhood of $\infty$ on which this formal series converges to an isomorphism. As in the introduction, we will take $f, g\in K[z]$ monic, with $\deg(g)<\deg(f)$.

Before continuing, however, we make a comment on what it means to evaluate $\xi$ at a point, since it is a power series in infinitely many (related) variables. As we saw in Section~\ref{sec:formal}, the monomials $w\in M\subseteq A$ are partially ordered by valuation, with the property that for any $X$ there are only finitely many $w$ with $v(w)\leq X$. Since each monomial involves only finitely many variables, each partial sum $\sum_{v(w)\leq X}c(w)w(P)$ makes sense, ultimately being just a polynomial in finitely many of the $w_i$. We can then reasonably define
\[\xi(P)=\lim_{X\to \infty}\sum_{v(w)\leq X} c(w)w(P),\]
if that limit exists (in whatever topology is natural for the context). Note that in the context of local fields, this limit exists if and only if $c(w)w(P)\to 0$ as $v(w)\to\infty$.
\begin{theorem}\label{th:isomdisk}
Let $\path$, $\mathscr{B}$, and $\Xi$ be as defined as in Section~\ref{sec:formal}. Then there exist neighbourhoods $\infty\in U_1\subseteq \path(K)$ and $\infty\in U_2\subseteq \mathscr{B}(K)$ such that $\Xi$ induces an isomorphism $\Xi:U_1\to U_2$. Furthermore, if the correspondence has good reduction, we may take $U_1=\pi^{-1}(D(\infty; 1))$ and $U_2=\psi^{-1}(D(\infty; 1))$.
\end{theorem}

\begin{proof}
The case of good reduction is not difficult. Since everything in Section~\ref{sec:formal} works over the ring $R$ of integers of $K$, we see that the coefficients $c(w)$ of the series $\xi$ are all integral. So the limit defining $\xi(P)$ exists precisely in case $|\pi\circ\sigma^n(P)|>1$ for all $n$. But under the hypothesis of good reduction, we also have \[\max\{|\pi\circ\sigma(P)|, 1\}^e=\max\{|\pi(P)|, 1\}^d\]
for all $P$, so the condition  that  $|\pi\circ\sigma^n(P)|>1$ for all $n$ is equivalent to the same condition just for $n=0$. Ergo if $U_1=\pi^{-1}(D(\infty; 1))$, then $\Xi$ converges on $U_1$. Applying the same logic to $\Psi$ demonstrates the claim in the case of good reduction.

We now address the general case. First, let $\mathscr{G}=\mathscr{B}\setminus\{0, \infty\}$, and for $\alpha=(\alpha_0,\alpha_1, ...)\in \mathscr{G}(K)$ let
\[\varphi_\alpha(w_0, w_1, ...)=(\alpha_0w_0, \alpha_1w_1, ...),\]
\[f_i(x)=x^d+\alpha_i a_{d-1}x^{d-1}+\cdots +\alpha_i^da_0,\]
and
\[g_i(y)=y^e+\alpha_{i} b_{e-1}y^{e-1}+\cdots +\alpha_i^eb_0.\]
Note that
\[g_{i+1}(\alpha_{i+1}w_{i+1})=f_i(\alpha_iw_i)\]
for all $i$ if and only if $g(w_{i+1})=f(w_i)$.

Let $\xi_\alpha$ be the series defined in the proof of Theorem~\ref{th:formal} relative to the sequences of polynomials $g_i$ and $f_i$. Note that if $\xi$ is the series defined relative to $g$ and $f$, then we have $\alpha_0\xi= \xi_{\alpha}\circ\varphi_\alpha(w)$. Also, note that we may induce the $g_i$ and $f_i$ to all have integral coefficients merely by requiring that $|\alpha_0|$ be small enough. In other words, there exists an $\epsilon>0$ such that $\xi$ converges at $(w_0, w_1, ...)$ with $|w_n|<\epsilon^{(d/e)^n}$. But for sufficiently small $\epsilon>0$ the condition $|w_0|<\epsilon$ already ensures this, so $\xi$ (and hence $\Xi$) is defined on $\pi^{-1}(D(\infty; \epsilon))$. The same reasoning applies to $\Psi$.
\end{proof}

\begin{remark}
The construction of the formal morphism $\Xi$ in Section~\ref{sec:formal} is relatively indifferent to the nature of the ground field, and so in particular over $\CC$ it gives a formal isomorphism between $g(y)=f(x)$ and $y^{\deg(g)}=x^{\deg(f)}$ at infinity, generalizing the B\"{o}ttcher coordinate. It would be interesting to know whether or not there exists a neighbourhood $U\subseteq \widehat{\CC}$ of infinity such that $\Xi$ defines a function on $\pi^{-1}(U)\subseteq \path_C(\CC)$. The author's initial investigations have failed to establish convergence in this setting.
\end{remark}


\section{The action of Galois}\label{sec:galois}

Note that the map $\Xi$ defined in Section~\ref{sec:formal} is automatically Galois equivariant, wherever it converges. If $K$ is a complete local field with absolute value $|\cdot|$, if $L/K$ is a (possibly infinite) Galois extension, and if $\sigma\in\Gal(L/K)$ then $|\sigma(x)|=|x|$ by the uniqueness of the extension of $|\cdot|$ to $L$. It follows that $\sigma$ is continuous, and so the fact that $\sigma$ commutes with finite sums implies it also commutes with (convergent) infinite sums.

The next lemma makes precise the claim that the action of $\operatorname{Gal}(\sep{K}/K)$ on preimages of $P\in\path(K)$ under $\sigma:\path\to\path$ relates in a straightforward way to the action of Galois on preimages of $\pi(P)$ in the original correspondence.

\begin{lemma}\label{lem:twopathspaces}
Let $P\in \path(K)$. Then there is a $\operatorname{Gal}(\sep{K}/K)$-equivariant isomorphism between $\check{\path}_{C}(\pi(P))$ and $\check{\path}_{(\path_C, \sigma)}(P)$.
\end{lemma}

\begin{proof}
It follows from the commutativity of~\eqref{eq:comdi} that the map
\[(..., Q_2, Q_1, Q_0)\mapsto (..., \pi(Q_2), \pi(Q_1), \pi(Q_0))\]
is an isomorphism between backward paths from $P$ in $\path$ and backward paths from $\pi(P)$ in $C$. Since $\pi:\path\to \PP^1$ is a morphism of $K$-schemes, this map is also $\Gal(\sep{K}/K)$-equivariant, given the condition that $P$ is $K$-rational.
\end{proof}

So now studying preimage graphs in $C$ can be related to studying the same in $\path$. On the other hand, our partial isomorphism from $\path$ to $\mathscr{B}$ relates this to preimages in $\mathscr{B}$, which in turn are related to preimages in $B$. But the Galois-theoretic behaviour of preimage trees in $B:y^e=x^d$ is fairly easy to understand, by standard multiplicative Kummer theory.

In the proof of Theorem~\ref{th:isomdisk} we made use of $\mathscr{G}=\mathscr{B}\setminus\{0, \infty\}$, which we point out here is an affine group scheme under coordinatewise multiplication. We will study the action of Galois on points in this algebraic group, subject to the constraint that $\gcd(e, d)=1$, but removing the constraint that $e<d$. As such, we may freely swap $e$ and $d$, and consider forward orbits instead of backward ones. For convenience, then, we will occasionally make reference to the exponents in the subscript:
\[\mathscr{G}_{e, d}=\operatorname{Spec}(K[x_0^{\pm 1}, x_1^{\pm 1}, ...:x_{i+1}^e=x_i^d\text{ for all }i\geq 0]).\]
Note that the morphism $\pi:\mathscr{G}\to \mathbb{G}_\mathrm{m}$ of $K$-schemes is in fact a group homomorphism, and so we have a short exact sequence of Galois modules
\[\xymatrix{
0 \ar[r]  & \mu_{e, d} \ar[r] &  \mathscr{G}_{e, d} \ar[r]^{\pi} & \mathbb{G}_{\mathrm{m}} \ar[r] & 0
}.\]
\begin{lemma}\label{lem:ede1}
Suppose that $\gcd(e, d)=1$. Then there is an isomorphism $\mathscr{G}_{e, d}\to \mathscr{G}_{e, 1}$ of $K$-schemes.
\end{lemma}

\begin{proof}
For $i\geq j\geq 0$, we choose integers $a_{i, j}$ and $b_{i, j}$ as follows. First, take $a_{0, 0}=1$ and $b_{0,0}=0$. Now, suppose that $a_{i, j}$ and $b_{i, j}$ have been chosen for all $0\leq j\leq i<I$. We set $b_{I, 0}=0$, and for each $0\leq j< I$ we choose $b_{I, j+1}$ and $a_{I, j}$ so that
\begin{equation}\label{eq:ab}
b_{I, j+1}d+a_{I, j}e=a_{I-1, j}+eb_{I, j},
\end{equation}
which we may always do, since $\gcd(e, d)=1$. Specifically, given a single solution $sd+et=1$, we may define
\begin{gather}
b_{I, j+1} = s(a_{I-1, j}+eb_{I, j})\label{eq:b}\\
a_{I, j} =  t(a_{I-1, j}+eb_{I, j}).\nonumber
\end{gather}
 Finally, we set $a_{I, I}=b_{I, I}$.

For each $i$, we set
\[w_i=\prod_{0\leq j\leq i} x_j^{a_{i, j}},\]
and claim that the map
$(x_0,  x_1, x_2, ...)\mapsto (w_0, w_1, w_2, ...)$
gives the desired isomorphism.

First we will verify that $w_{n+1}^e=w_n$ for every $n$. To see this, note that
\begin{equation}\label{eq:zetaomega}
\frac{w_{n}^e}{w_{n-1}}=x_n^{a_{n, n}e}x_{n-1}^{a_{n, n-1}e-a_{n-1, n-1}}\cdots x_{0}^{a_{n, 0}e-a_{n-1, 0}-b_{n, 0}e},
\end{equation}
where the apparently extraneous $b_{n, 0}$ fails to falsify the equality because $b_{n, 0}=0$. Now, by the definition of the $a_{i, j}$ and $b_{i, j}$, and the relation $x_{j+1}^e=x_j^d$, we have for each $0\leq j\leq I$,
\begin{eqnarray*}
x_{j+1}^{a_{n, j+1}e-a_{n-1, j+1}}x_{j}^{a_{n, j}e-a_{n-1, j}-b_{n, j}e}&=&
x_{j+1}^{a_{n, j+1}e-a_{n-1, j+1}}x_{j}^{-b_{n, j+1}d}\\
&=&x_{j+1}^{a_{n, j+1}e-a_{n-1, j+1}-b_{n, j+1}e}.
\end{eqnarray*}
Combining this with \eqref{eq:zetaomega} for $j=0, 1, 2, ..., i-1$ we are left with
\[\frac{w_{n}^e}{w_{n-1}}=x_{n}^{a_{n, n}e-b_{n, n}e}=1,\]
since $a_{n, n}=b_{n, n}$. 

So we have confirmed that the $w_i$ describe a morphism of $K$-schemes $\varphi:\mathscr{G}_{e, d}\to \mathscr{G}_{e, 1}$. We next show that this morphism is injective. Note that if $\pi_{e,d }$ is projection of $\mathscr{G}_{e, d}$ onto the first coordinate, then $\pi_{e, d}=\pi_{e, 1}\circ\varphi$ (since $a_{0, 0}=1$) and so $\ker(\varphi)\subseteq \mu_{e, d}=\ker(\pi_{e, d})$.
So now suppose that $(1, 1, ..., \zeta_n, ...)\in \ker(\varphi)$. By definition, we have
\[\varphi(1, 1, ..., \zeta_n, ...)= (1, 1, ..., \zeta_n^{a_{n, n}}, ...),\]
and so $\zeta_n^{a_{n, n}}=1$. Since we also have $\zeta_n^e=1^d=1$, we may conclude that $\zeta_n=1$ as soon as we know that $\gcd(a_{n, n}, e)=1$.  But by the definition~\eqref{eq:b} we have
\[a_{n, n}=b_{n, n}\equiv sa_{n-1, n-1}\equiv s^n\mod{e},\]
for all $n\geq 1$, and $s$ is coprime to $e$.
So in fact, the kernel of $\varphi$ is trivial.

Finally, we will show that $\varphi$ is surjective, by showing that the truncated map
\[\phi_n(x_0, ..., x_n)=(w_0, ..., w_n)\]
is surjective for each $n$. For $n=0$ this is obvious, since the truncated map is then the identity on $\mathbb{G}_\mathrm{m}$. Suppose that the claim is true for $n=k$, and consider a sequence $(w_0, ..., w_{k+1})$. By the induction hypothesis, we may solve $\varphi_k(x_0, ..., x_k)=(w_0, ..., w_k)$. But then $\varphi_{k+1}$ maps the $e$ distinct points of the form $(x_0, ..., x_k, y)$ injectively to the $e$ distinct points of the form $(w_0, ..., w_k, z)$, and so this association must also be surjective. This confirms that $\varphi$ is an isomorphism.
\end{proof}

%

We now assume that all $(ed)^n$th roots of unity are $K$-rational. For the proof of Theorem~\ref{th:main}, we will also need to extend $K$ so that there exists a $P\in\path(K)$ such that $\alpha=\pi(P)$. This cannot be done in a finite extension, or indeed even in an extension which leaves the valuation discrete, but luckily it can be done in a way which makes the value group a subgroup of $\ZZ[\frac{1}{e}]$.
\begin{lemma}\label{lem:kummer}
Let $L/K$ be an algebraic extension, suppose  the valuation group $\Gamma=v(L^*)\subseteq \QQ$ contains reciprocals of none of the prime divisors of $d$, and 
let $\alpha\in \mathbb{G}_\mathrm{m}(L)$ with $|\pi(\alpha)|>1$. Then $\Gal(\sep{L}/L)$ acts on $P_{B_{1, d}, \alpha}$ as a finite-index subgroup of a $d$-adic group.
\end{lemma}

\begin{proof}
Let $\mu[d]$ denote the group of $d$th roots of unity. As usual, if $\beta^d=\alpha\in L$, then $\sigma\mapsto \sigma(\beta)/\beta$ gives a map $\Gal(L(\beta)/L)\to\mu[d]$ whose image is isomorphic to the subgroup of $L^*/(L^*)^{d}$ generated by $\alpha$. Suppose that the numerator of $v(\alpha)$ is prime to $d$. Then composing with the map $L^*/(L^*)^d\to \Gamma/d\Gamma$, we see that $\alpha$ has order $d$ in $L^*/(L^*)^d$. Replacing $d$ with $d^n$, and letting $n\to\infty$, we have in this case that the natural map $\Gal(\sep{L}/L)\to \mu[d^\infty]\cong \widehat{\ZZ}_d$ is surjective, and hence $\Gal(\sep{L}/L)$ acts on $P_{B_{1, d}, \alpha}$ as a $d$-adic subgroup.

In general, suppose that $v(\alpha)\neq 0$. Then for some $m$, the numerator of $v(\alpha)/d^m$ is prime to $d$. By the argument above, for any $\beta^{d^m}=\alpha$, $\Gal(\sep{L}/L(\beta))$ acts on $P_{B_{1, d}, \beta}$ as a $d$-adic subgroup of $\Aut(P_{B_{1, d}, \beta})$. On the other hand, if $\beta'$ is another $d^m$th root of $\alpha$, note that $\beta'=\zeta \beta$ for so $d^m$th root of unity $\zeta\in K$, and so $L(\beta')=L(\beta)$. Furthermore, multiplication by any $d^n$th root of $\zeta$ takes the $n$th level of the tree $P_{B_{1, d}, \beta}$ to the $n$th level of the tree $P_{B_{1, d}, \beta'}$, and so the action of $\Gal(\sep{L}/L(\beta))$ on these two sub-trees of $P_{B_{1, d}, \alpha}$ is compatible. So up to labelling, the image of $\Gal(\sep{L}/L)$ in $\Aut(P_{B_{1, d}, \alpha})$ contains at least the full subgroup which acts trivially on the first $m+1$ levels.
\end{proof}

We now turn to the proof of the main theorem. For simplicity, if $G$ is a group acting on two sets $X$ and $Y$, we say that the actions are isomorphic if and only if there is a bijection $\varphi:X\to Y$ such that $\varphi(x^\sigma)=\varphi(x)^\sigma$ for every $\sigma\in G$ and $x\in X$.

\begin{proof}[Proof of Theorem~\ref{th:second}] Rather than prove the theorem in the case that $v(K^*)\cong \ZZ$, we relax the hypothesis to assume that no prime divisor of $d$ has a reciprocal in $v(K^*)$.
Now note that if $L/K$ is an algebraic extension, then $\Gal(\sep{L}/L)\subseteq \Gal(\sep{K}/K)$, and so it suffices to prove the theorem after an algebraic extension of the base. In particular, with $C$ and $\alpha$ as in the theorem, we will assume without loss of generality that $K$ contains all $(ed)^n$th roots of unity, and that there is exists a $P\in \path(K)$ with $\alpha=\pi(P)$. Note that this is possible after an algebraic extension of the base, but certainly not a finite one. Also note that $e$ becomes invertible in the value group of this extension, but since $\gcd(e, d)=1$, our property that no prime divisor of $d$ has a reciprocal appearing in the value group is maintained. We may also pass to a finite extension such that the now-isomorphic correspondence $g(\gamma y)=f(\gamma x)$ has leading coefficients of $f$ and $g$ the same, and hence equal to 1 without loss of generality.

Now, with these assumptions in place, note that the action of $\Gal(\sep{K}/K)$ on the preimage tree of $\alpha$ under $C$ is isomorphic to that on the preimage tree of $P$ under $\sigma:\path\to\path$ by Lemma~\ref{lem:twopathspaces}. By Theorem~\ref{th:isomdisk} there is an isomorphism from this to the action of $\Gal(\sep{K}/K)$ on the preimage tree of $\Xi(P)\in\mathscr{G}_{e, d}(K)$.  Lemma~\ref{lem:twopathspaces} now shows that the action of $\Gal(\sep{K}/K)$ on this is isomorphic to its action on the space of backward orbits of $\beta=\psi\circ\Xi(P)$ under the correspondence $B_{e, d}$. So it suffices to restrict attention to the case of correspondences of the form $B_{e, d}:y^e=x^d$. Note that backward orbits in $B_{e, d}$ correspond to forward orbits in $B_{d, e}$; by Lemma~\ref{lem:ede1}, it suffices to restrict further to the case $e=1$, which is handled by Lemma~\ref{lem:kummer}.
\end{proof}

As noted in the introduction, Theorem~\ref{th:main} follows directly from Theorem~\ref{th:second}.


\end{document}